\documentclass[reqno,12pt]{amsart}

\usepackage[
  letterpaper,
  textwidth=15cm,
  textheight=22cm,
  top=1in,
  hcentering
]{geometry}

\usepackage{amsmath,amssymb,amsthm,mathtools}
\usepackage{microtype}
\usepackage{xcolor}
\usepackage{hyperref}
\usepackage[nameinlink,capitalize]{cleveref}

\hypersetup{
  colorlinks=true,
  linkcolor=blue!60!black,
  citecolor=blue!60!black,
  urlcolor=blue!60!black,
  pdftitle={Perturbed Polynomial Powers and Bourgain Entropy Obstructions for Khintchin Averages},
  pdfauthor={Zhuowen Guo and Shuhao Zhang}
}

\newtheorem{theorem}{Theorem}[section]
\newtheorem{corollary}[theorem]{Corollary}
\newtheorem{proposition}[theorem]{Proposition}
\newtheorem{lemma}[theorem]{Lemma}

\theoremstyle{definition}
\newtheorem{definition}[theorem]{Definition}

\numberwithin{equation}{section}

\crefname{theorem}{Theorem}{Theorems}
\Crefname{theorem}{Theorem}{Theorems}
\crefname{proposition}{Proposition}{Propositions}
\Crefname{proposition}{Proposition}{Propositions}
\crefname{corollary}{Corollary}{Corollaries}
\Crefname{corollary}{Corollary}{Corollaries}
\crefname{lemma}{Lemma}{Lemmas}
\Crefname{lemma}{Lemma}{Lemmas}
\crefname{definition}{Definition}{Definitions}
\Crefname{definition}{Definition}{Definitions}
\crefname{section}{Section}{Sections}
\Crefname{section}{Section}{Sections}

\newcommand{\T}{\mathbb T}
\newcommand{\N}{\mathbb N}
\newcommand{\Z}{\mathbb Z}
\newcommand{\R}{\mathbb R}
\newcommand{\cQ}{\mathcal Q}
\newcommand{\PFNP}{\mathsf{PFNP}}
\newcommand{\ee}{\mathrm e}
\newcommand{\eps}{\varepsilon}
\newcommand{\abs}[1]{\left\lvert #1\right\rvert}
\newcommand{\norm}[1]{\left\lVert #1\right\rVert}
\newcommand{\Ent}{\operatorname{Ent}}

\title[Perturbed polynomial powers and Khintchin averages]{Perturbed Polynomial Powers and Bourgain Entropy Obstructions for Khintchin Averages}

\author{Zhuowen Guo}
\address{School of Mathematical Sciences, University of Science and Technology of China, Hefei, Anhui 230026, P.R. China}
\email{guozw0920@mail.ustc.edu.cn}

\author{Shuhao Zhang}
\address{School of Mathematical Sciences, University of Science and Technology of China, Hefei, Anhui 230026, P.R. China}
\email{yichen12@mail.ustc.edu.cn}

\date{June 10, 2026}

\subjclass[2020]{Primary 37A45; Secondary 11K06, 11B83, 42A61}
\keywords{Khintchin averages, Bourgain entropy criterion, translated powers, primitive divisors, polynomial powers, uniform distribution}

\begin{document}

\begin{abstract}
Let \(a\ge2\), let \(p\in\Z[n]\) be eventually increasing and eventually non-negative, and let
\(\lambda_n=a^{p(n)}+f(n)>0\), where \(f(n)\in\Z\).  We prove, using
Bourgain's bounded entropy criterion, that if
\(\log(1+f(n)a^{-p(n)})\) is eventually non-zero and decays geometrically
in absolute value, then \((\lambda_n)\) is neither \(L^\infty\)-Khintchin nor
\(L^1\)-Khintchin.

In particular, for every \(c\in\Z\setminus\{0\}\), every positive tail of
\((a^{p(n)}+c)_{n\ge1}\) is non-Khintchin.  The same conclusion applies to
the standard examples \(a^n+c\), \(a^n+b^n\), and, whenever eventually
positive, \(a^n-b^n\), with \(a\neq b\).  Thus these perturbations of
geometric powers lie on the unstable side of the Khintchin problem.  This
gives a negative answer, in the translated-power case, to the question of
Fan--Fan--Queff\'elec--Queff\'elec on the stability of translated powers.
\end{abstract}

\maketitle

\section{Introduction}
We study the almost-everywhere convergence of the averages \(A_N^\lambda f\)
for measurable functions \(f\), especially at the endpoint spaces
\(L^\infty\) and \(L^1\). This problem is closely related to Weyl's metric
theorem.

Our main object is the circle group \(\mathbb T=\mathbb R/\mathbb Z\),
equipped with its Haar probability measure \(m\). For a sequence
\(\lambda=(\lambda_n)_{n\ge 1}\) of positive integers, define
\[
  A_N^\lambda f(x)=\frac1N\sum_{n\le N} f(\lambda_n x),
  \qquad x\in\mathbb T.
\]
We say that \(\lambda\) is \(L^p\)-Khintchin, \(1\le p\le \infty\), if
\[
  A_N^\lambda f(x)\longrightarrow \int_{\mathbb T} f\,dm
  \quad\text{for a.e. }x\in\mathbb T
\]
for every \(f\in L^p(\mathbb T)\).

The study of \(L^p\)-Khintchin sequences, and in particular of the endpoint
case \(p=1\), is a natural and delicate problem. By Weyl's metric theorem  \cite{Weyl1916}, if
the integers \(\lambda_n\) are distinct, then the sequence \((\lambda_n x)\)
is uniformly distributed modulo one for a.e. \(x\). Hence the above convergence
holds for every continuous function, and more generally for every Riemann
integrable function. Khintchine \cite{Khintchine1923} asked whether, for the special
sequence \(\lambda_n=n\), the same conclusion remains valid for characteristic
functions of measurable sets. Marstrand \cite{Marstrand1970} disproved this
and, more generally, exhibited broad families of sequences for which the
endpoint assertion fails, including polynomial sequences, sequences with
power-law counting function, and pairwise coprime sequences.

There are also positive examples.  For geometric powers \(\lambda_n=q^n\), the averages are Birkhoff averages for the map \(x\mapsto qx\) on \(\T\).  Raikov and Riesz  \cite{Raikov1936,Riesz1945} proved the corresponding \(L^1\) convergence.  More generally, Marstrand and later Nair \cite{Marstrand1970,Nair1990} studied positive results for finite-rank multiplicative semigroups .  These examples show that lacunarity alone is not the issue.  The geometric sequence \((2^n)\) is positive, while the translated sequence \((2^n+1)\) behaves very differently.

Fan, Fan, Queffelec and Queffelec \cite{FanFanQueffelecQueffelec2026} formulated the modern Khintchin class of a sequence and asked whether the Khintchin property is stable under translation; the sequence \((2^n+1)\) is the basic test case.  We answer this in the negative as part of a more general statement for perturbed polynomial powers.

\begin{theorem}[perturbed polynomial powers]\label{thm:perturbed-polynomial-powers}
Let \(a\ge2\), let \(p\in\Z[n]\) be eventually increasing and eventually non-negative, and let
\[
  \lambda_n=a^{p(n)}+f(n)>0,\qquad f(n)\in\Z.
\]
Put
\[
  \theta_n=\log\left(1+f(n)a^{-p(n)}\right).
\]
If
\[
  0<\abs{\theta_{n+1}}\le \rho \abs{\theta_n}\qquad(n\ge n_0)
\]
for some \(0<\rho<1\), then \((\lambda_n)\) is neither \(L^\infty\)-Khintchin nor \(L^1\)-Khintchin.

In particular, for every \(c\in\Z\setminus\{0\}\), every positive tail of $\lambda_n=a^{p(n)}+c$
is non-Khintchin.
\end{theorem}

 The theorem includes, for instance, the translated powers \((2^n+1)\), as well
as \((2^{n^2}+1)\), \((a^n+b^n)\), \((a^n-b^n)\), and \((a^{p(n)}-1)\), after
passing to a positive tail when necessary.

The proof combines Bourgain's bounded entropy theorem \cite{Bourgain1988} with an arithmetic
product-frequency construction.  The entropy theorem reduces non-Khintchinness
to producing, for arbitrarily large \(d\), one trigonometric polynomial whose
averaging orbit contains \(d\) points separated in \(L^2(\T)\).  The
product-frequency condition supplies such separation by placing Boolean layers
of the products \(\Lambda_S=\prod_{n\in S}\lambda_n\) into Fourier
frequencies.  We verify this condition either by Bang--Zsigmondy primitive
divisors, for the binomial families, or by a logarithmic separation argument,
for perturbed polynomial powers.

Thus the theorem is not merely a statement about growth.  It shows that the
multiplicative Fourier structure behind the positive sequence \(a^{p(n)}\)
can be destroyed by small integer perturbations.

The paper is organized as follows.  \Cref{sec:bourgain} recalls Bourgain entropy and reduces the problem to separated \(L^2\)-orbits.  \Cref{sec:PFNP} states \(\PFNP\) and proves directly that it implies non-Khintchinness.  \Cref{sec:primitive} records primitive-divisor verifications.  \Cref{sec:log} proves the logarithmic criterion and then proves \cref{thm:perturbed-polynomial-powers}.  

\section{Bourgain entropy and separated orbits}\label{sec:bourgain}

For a subset \(B\) of a Hilbert space \(H\), let \(\Ent_\delta(B;H)\) be the least number of open \(H\)-balls of radius \(\delta\) needed to cover \(B\).  We use the following form of Bourgain's bounded entropy theorem.

\begin{theorem}[Bourgain entropy criterion]\label{thm:bourgain}
Let \((X,\mu)\) be a probability space.  Let \((T_j)_{j\ge1}\) be positive commuting contractions on \(L^1(X)\) whose restrictions to \(L^2(X)\) are isometries.  Assume that \(T_j1=1\) and that
\[
  \frac1J\sum_{j\le J}T_jg\longrightarrow \int_X g\,d\mu
  \quad\text{in }L^1(X),\qquad g\in L^1(X).
\]
Let each \(S_N\) be a finite convex combination of the \(T_j\).  If \(S_Nf\) converges a.e. for every \(f\in L^\infty(X)\), then for every \(\delta>0\) there is a constant \(C(\delta)<\infty\) such that
\[
  \Ent_\delta(\{S_Nf:N\ge1\};L^2(X))\le C(\delta)
\]
for every \(f\in L^\infty(X)\) with \(\norm{f}_2\le1\).
\end{theorem}

On \(\T\), let
\[
  T_jf(x)=f(jx),\qquad j\ge1.
\]
The map \(x\mapsto jx\) preserves Haar measure.  Hence \(T_j\) is positive, \(T_j1=1\), and \(T_j\) is an \(L^2\)-isometry.  The operators commute because \(T_jT_k=T_{jk}\).  If \(e_k(x)=\ee^{2\pi i kx}\) and \(k\ne0\), then
\[
  \left\|\frac1J\sum_{j\le J}T_je_k\right\|_2^2
  =\frac1{J^2}\sum_{j,\ell\le J}\int_\T \ee^{2\pi i k(j-\ell)x}\,dm(x)
  =\frac1J.
\]
The density of trigonometric polynomials in \(L^1(\T)\) gives the averaging hypothesis in \cref{thm:bourgain}.  Therefore
\[
  A_N^\lambda=\frac1N\sum_{n\le N}T_{\lambda_n}
\]
is an admissible sequence of convex combinations.

\begin{proposition}[entropy witness]\label{prop:entropy-witness}
Let \(\lambda=(\lambda_n)\) be a sequence of positive integers.  Suppose that there is \(\delta_0>0\) such that for every \(d\ge1\) we can find a trigonometric polynomial \(f_d\), \(\norm{f_d}_2=1\), and integers
\[
  1\le N_1<\cdots<N_d
\]
for which
\[
  \norm{A_{N_i}^\lambda f_d-A_{N_j}^\lambda f_d}_2\ge\delta_0
  \qquad(i\ne j).
\]
Then \(\lambda\) is neither \(L^\infty\)-Khintchin nor \(L^1\)-Khintchin.
\end{proposition}

\begin{proof}
Assume that \(A_N^\lambda f\) converges a.e. for every bounded \(f\).  Apply \cref{thm:bourgain} with \(S_N=A_N^\lambda\) and \(\delta=\delta_0/3\).  We get a number \(C(\delta)\) that covers every bounded \(L^2\)-unit orbit.  Choose \(d>C(\delta)\).  The orbit of \(f_d\) contains \(d\) points at pairwise distance at least \(\delta_0\), so one open ball of radius \(\delta\) contains at most one of them.  Its covering number is at least \(d\), contradiction.  Thus bounded convergence fails for some trigonometric polynomial.  Since bounded functions are in \(L^1(\T)\), \(L^1\)-Khintchinness fails as well.
\end{proof}

We will discard finitely many initial labels several times.  This does not change the conclusion.

\begin{lemma}[finite-prefix invariance]\label{lem:finite-prefix}
Let \(h\ge0\) and set \(\lambda_n^{[h]}=\lambda_{n+h}\).  Then \(\lambda\) is \(L^p\)-Khintchin if and only if \(\lambda^{[h]}\) is \(L^p\)-Khintchin, for \(p=1\) and for \(p=\infty\).
\end{lemma}

\begin{proof}
For bounded \(f\), the first \(h\) terms contribute \(O_f(N^{-1})\) pointwise to the average.  The denominator change from \(N\) to \(N+h\) tends to one.  For \(f\in L^1\), each fixed pullback \(f(\lambda_jx)\) is finite for a.e. \(x\), because \(x\mapsto \lambda_jx\) preserves null sets.  The same argument applies.
\end{proof}

\section{The product-frequency criterion}\label{sec:PFNP}

We now state the only auxiliary arithmetic condition used in the entropy construction.  For a finite set \(S\subset\N\), write $  \Lambda_S=\prod_{n\in S}\lambda_n,$
with \(\Lambda_\varnothing=1\).

\begin{definition}[product-frequency condition]\label{def:PFNP}
A sequence \(\lambda=(\lambda_n)\) of positive integers has \(\PFNP\) if, after deleting finitely many initial terms, there are \(\eta>0\), an unbounded set \(\cQ\subset\N\), and sets
\[
  E_Q\subset (Q/2,Q]\cap\N,\qquad \#E_Q\ge \eta Q\quad(Q\in\cQ),
\]
with the following property.  Whenever \(Q_1<\cdots<Q_d\) are in \(\cQ\) and \(Q_{a+1}>4Q_a\), put \(E_a=E_{Q_a}\).  Then:

First, products on equal layers are unique.  If \(S_a,T_a\subset E_a\), \(\#S_a=\#T_a\), and
\[
  \prod_{a=1}^d \Lambda_{S_a}=\prod_{a=1}^d \Lambda_{T_a},
\]
then \(S_a=T_a\) for every \(a\).

Second, a prefix label can raise one layer only in the forced way.  If \(i\le d\), \(m\le Q_i\), \(S_a,T_a\subset E_a\),
\[
  \#T_i=\#S_i+1,
  \qquad
  \#T_a=\#S_a\quad(a\ne i),
\]
and
\[
  \lambda_m\prod_{a=1}^d\Lambda_{S_a}
  =\Lambda_{T_i}\prod_{a\ne i}\Lambda_{T_a},
\]
then
\[
  m\in E_i\setminus S_i,
  \qquad
  T_i=S_i\cup\{m\},
  \qquad
  T_a=S_a\quad(a\ne i).
\]
\end{definition}

The second clause says that the only way to enter the next layer in the \(i\)-th block is to add the same index \(m\).

\begin{theorem}[product-frequency entropy criterion]\label{thm:PFNP-obstruction}
If \(\lambda\) has \(\PFNP\), then \(\lambda\) is neither \(L^\infty\)-Khintchin nor \(L^1\)-Khintchin.
\end{theorem}
 \begin{proof}
By Lemma~\ref{lem:finite-prefix}, we may pass to a tail of the sequence
and assume that Definition~\ref{def:PFNP} holds. Fix \(d\geq 1\), and
choose \(Q_1<\cdots<Q_d\) in \(\mathcal Q\) such that \(Q_{a+1}>4Q_a\).
For \(a=1,\ldots,d\), write \(E_a=E_{Q_a}\), \(M_a=\#E_a\), and
\(r_a=\lfloor M_a/2\rfloor\).

Define
\[
  \Omega_0=\left\{\prod_{a=1}^d \Lambda_{S_a}:
  S_a\subset E_a,\ \#S_a=r_a\right\}.
\]
By the first clause of \(\operatorname{PFNP}\), each element of
\(\Omega_0\) has a unique representation of this form. Hence
\(f_d(x):=\#\Omega_0^{-1/2}\sum_{\omega\in\Omega_0} e^{2\pi i\omega x}\)
satisfies \(\|f_d\|_2=1\).

For each \(i=1,\ldots,d\), define
\[
  \Omega_i=\left\{\Lambda_{T_i}\prod_{a\ne i}\Lambda_{T_a}:
  T_i\subset E_i,\ \#T_i=r_i+1,\ 
  T_a\subset E_a,\ \#T_a=r_a\ \text{for }a\ne i\right\}.
\]
Let \(\Pi_i\) be the orthogonal projection in \(L^2(\mathbb T)\) onto the
closed span of \(\{e^{2\pi i\omega x}:\omega\in\Omega_i\}\). Again, the
first clause of \(\operatorname{PFNP}\) gives unique representations for
the elements of \(\Omega_i\).

We now compute the projection of \(A_{Q_i}^{\lambda}f_d\). Fix
\(\omega=\Lambda_{T_i}\prod_{a\ne i}\Lambda_{T_a}\in\Omega_i\). The
Fourier coefficient of \(\Pi_i A_{Q_i}^{\lambda}f_d\) at frequency
\(\omega\) is
\[
  \frac{1}{Q_i\#\Omega_0^{1/2}}
  \#\left\{(m,S_1,\ldots,S_d):
  m\leq Q_i,\ \#S_a=r_a,\ 
  \lambda_m\prod_{a=1}^d\Lambda_{S_a}=\omega
  \right\}.
\]
By the second clause of \(\operatorname{PFNP}\), every solution is forced:
\(m\in T_i\), \(S_i=T_i\setminus\{m\}\), and \(S_a=T_a\) for \(a\ne i\).
Conversely, each \(m\in T_i\) gives such a solution. Thus the number of
solutions is \(r_i+1\), and every target frequency has coefficient
\((r_i+1)/(Q_i\#\Omega_0^{1/2})\). Therefore
\[
\begin{aligned}
  \|\Pi_iA_{Q_i}^\lambda f_d\|_2^2
  &= \#\Omega_i\frac{(r_i+1)^2}{Q_i^2\#\Omega_0}  \\
  &= \frac{(r_i+1)^2}{Q_i^2}
     \frac{\binom{M_i}{r_i+1}}{\binom{M_i}{r_i}} \\
  &= \frac{(r_i+1)(M_i-r_i)}{Q_i^2}.
\end{aligned}
\]
Since \(r_i=\lfloor M_i/2\rfloor\), we have
\((r_i+1)(M_i-r_i)\ge M_i^2/4\). Since \(M_i\ge\eta Q_i\), it follows
that \(\|\Pi_iA_{Q_i}^\lambda f_d\|_2\ge \eta/2\).

Now let \(\ell<i\). If a frequency in \(\Omega_i\) appeared in
\(A_{Q_\ell}^\lambda f_d\), then the same product identity would hold
with \(m\le Q_\ell\). The second clause of \(\operatorname{PFNP}\) would
force \(m\in E_i\subset(Q_i/2,Q_i]\). But
\(m\le Q_\ell\le Q_{i-1}<Q_i/4\), a contradiction. Hence
\(\Pi_iA_{Q_\ell}^\lambda f_d=0\) for every \(\ell<i\).

For \(\ell<i\), we have
\[
  \|A_{Q_i}^\lambda f_d-A_{Q_\ell}^\lambda f_d\|_2
  \ge \|\Pi_i(A_{Q_i}^\lambda f_d-A_{Q_\ell}^\lambda f_d)\|_2
  =\|\Pi_iA_{Q_i}^\lambda f_d\|_2
  \ge \eta/2.
\]
Thus the orbit of \(f_d\) contains \(d\) points separated by \(\eta/2\).
Since \(d\) is arbitrary, Proposition~\ref{prop:entropy-witness} applies.
\end{proof}

\section{Primitive-divisor verifications}\label{sec:primitive}

We record a short arithmetic route to \(\PFNP\).  It is useful for the translated-power example \((2^n+1)\) and for binomial families.

\begin{proposition}[primitive divisors imply the product-frequency condition]\label{prop:primitive-PFNP}
Assume that, after deleting finitely many initial terms, each \(\lambda_n\) has a prime divisor \(\pi_n\) such that
\[
  \pi_n\mid \lambda_n,
  \qquad
  \pi_n\nmid \lambda_m\quad(m<n).
\]
Then \(\lambda\) has \(\PFNP\).  Hence \(\lambda\) is neither \(L^\infty\)-Khintchin nor \(L^1\)-Khintchin.
\end{proposition}
\begin{proof}
By Lemma~\ref{lem:finite-prefix}, assume the primitive-divisor condition holds for all $n$. Let $\prod_{n\in F}\lambda_n^{\alpha_n} = \prod_{n\in F}\lambda_n^{\beta_n}$ with $\alpha_n, \beta_n \in \mathbb{Z}_{\ge 0}$. If $\alpha \neq \beta$, let $N = \max\{n \in F : \alpha_n \neq \beta_n\}$. Assuming $\alpha_N > \beta_N$ without loss of generality, canceling common factors yields
\[
\lambda_N^{\alpha_N - \beta_N} \prod_{n < N} \lambda_n^{\alpha_n'} = \prod_{n < N} \lambda_n^{\beta_n'}
\]
where $\alpha_n', \beta_n' \in \mathbb{Z}_{\ge 0}$. Since $\pi_N \mid \lambda_N$, it follows that $\pi_N \mid \prod_{n < N} \lambda_n^{\beta_n'}$. Since $\pi_N$ is prime, there exists $m < N$ such that $\pi_N \mid \lambda_m$, which contradicts the assumption that $\pi_N \nmid \lambda_m$ for all $m < N$. Thus $\alpha_n = \beta_n$ for all $n \in F$.

Set $\mathcal{Q} = \mathbb{N}$ and $E_Q = (Q/2,Q]\cap\mathbb{N}$. For $Q_1<\cdots<Q_d$ with $Q_{a+1}>4Q_a$, the blocks $E_a = E_{Q_a}$ are pairwise disjoint.
\begin{enumerate}
\item $\prod_{a=1}^d \Lambda_{S_a}=\prod_{a=1}^d \Lambda_{T_a} \implies \bigsqcup_{a=1}^d S_a = \bigsqcup_{a=1}^d T_a$. Intersecting both sides with $E_a$ yields $S_a = T_a$ for all $1 \le a \le d$.

\item $\lambda_m\prod_{a=1}^d\Lambda_{S_a} = \Lambda_{T_i}\prod_{a\ne i}\Lambda_{T_a} \implies \{m\}\sqcup \bigsqcup_{a=1}^d S_a = T_i\sqcup\bigsqcup_{a\ne i}T_a$.

For any $a \ne i$, intersecting both sides with $E_a$ yields $(\{m\} \cap E_a) \sqcup S_a = T_a$. Since $\#S_a = \#T_a$, it forces $\{m\} \cap E_a = \varnothing$, hence $m \notin E_a$ and $S_a = T_a$.

For $a = i$, intersecting both sides with $E_i$ yields $(\{m\} \cap E_i) \sqcup S_i = T_i$. Since $\#T_i = \#S_i + 1$, it forces $\{m\} \cap E_i = \{m\}$, hence $m \in E_i$. It follows that $m \notin S_i$ and $T_i = S_i \cup \{m\}$.
\end{enumerate}
Thus $\lambda$ satisfies $\PFNP$, and Theorem~\ref{thm:PFNP-obstruction} completes the proof.
\end{proof}

\begin{corollary}[coprime labels]\label{cor:coprime}
If \(\lambda_n>1\) and \(\gcd(\lambda_m,\lambda_n)=1\) for \(m\ne n\), then \(\lambda\) is neither \(L^\infty\)-Khintchin nor \(L^1\)-Khintchin.  In particular, the prime labels \(\lambda_n=p_n\) are non-Khintchin in this circle-endomorphism problem.
\end{corollary}

We use the following classical theorem of Bang and Zsigmondy
\cite{Bang1886,Zsigmondy1892}.

\begin{theorem}[Bang--Zsigmondy]\label{thm:bang-zsigmondy}
Let \(a>b\ge1\), \(\gcd(a,b)=1\), and \(r>1\).  Then \(a^r-b^r\) has a prime divisor that divides no \(a^s-b^s\) with \(1\le s<r\), except for the standard cases \((a,b,r)=(2,1,6)\) and \(r=2\) with \(a+b\) a power of \(2\).
\end{theorem}

\begin{corollary}[binomial families]\label{cor:binomial}
Let \(a>b\ge1\) and \(\gcd(a,b)=1\).  After deleting finitely many initial terms, each of the sequences
\[
  a^n-b^n,
  \qquad
  a^n+b^n
\]
is neither \(L^\infty\)-Khintchin nor \(L^1\)-Khintchin.  In particular, \((2^n+1)\) is not Khintchin.
\end{corollary}
 \begin{proof}
For \(a^n-b^n\), \cref{thm:bang-zsigmondy} gives, for all sufficiently
large \(n\), a prime \(\pi_n\mid a^n-b^n\) which divides no
\(a^m-b^m\) with \(m<n\).  Hence Proposition~\ref{prop:primitive-PFNP}
applies.

For \(a^n+b^n\), apply \cref{thm:bang-zsigmondy} to
\(a^{2n}-b^{2n}\).  For all sufficiently large \(n\), choose a prime
\(\pi_n\mid a^{2n}-b^{2n}\) which divides no \(a^s-b^s\) with \(s<2n\).
If \(\pi_n\mid b\), then \(\pi_n\mid a\), contradicting
\(\gcd(a,b)=1\).  Hence \(ab^{-1}\) is well defined modulo \(\pi_n\).
Moreover \(ab^{-1}\) has order \(2n\), since any smaller positive order
\(s<2n\) would imply \(\pi_n\mid a^s-b^s\).  Thus
\[
  (ab^{-1})^n\equiv -1\pmod{\pi_n}.
\]
Therefore \(\pi_n\mid a^n+b^n\).  If \(m<n\) and
\(\pi_n\mid a^m+b^m\), then \((ab^{-1})^{2m}\equiv1\pmod{\pi_n}\),
contradicting that the order is \(2n\).  Hence \(a^n+b^n\) satisfies the
primitive-divisor condition after deleting finitely many initial terms.
Proposition~\ref{prop:primitive-PFNP} gives the conclusion.
\end{proof}

 \section{Logarithmic perturbations}\label{sec:log}

We next show that \(L^\infty\)-goodness of exponential sequences is unstable
under additive perturbations.  Although \(a^{p(n)}\) is \(L^\infty\)-good in
many cases, for \(\lambda_n=a^{p(n)}+f(n)\), writing
\(\log\lambda_n=\nu_n\log a+\theta_n\) with \(\nu_n\in\mathbb Z\), we prove
that, under suitable assumptions, a logarithmic error \(\theta_n\) which does
not decay too fast on a positive-density set already forces
\((\lambda_n)\) to fail \(L^\infty\)-goodness.

\begin{proposition}[logarithmic criterion for \(\PFNP\)]\label{prop:log-PFNP}
Let \(a\ge2\), and let \((\lambda_n)\) be positive integers such that
\[
  \log\lambda_n=\nu_n\log a+\theta_n,
  \qquad \nu_n\in\Z,
  \qquad \theta_n\in\R.
\]
Assume that there are \(N_0\ge1\), \(\eta>0\), an unbounded set \(\cQ\subset\N\), and a set \(P\subset\{n\ge N_0\}\) such that
\[
  \#(P\cap(Q/2,Q])\ge\eta Q\qquad(Q\in\cQ),
\]
and
\begin{align}
  2\sum_{n\ge N_0}\abs{\theta_n} &<\log a, \label{eq:log-small}\\
  \sum_{\substack{p\in P\\p>n}}\abs{\theta_p} &<\abs{\theta_n}
       \qquad(n\ge N_0,\ n\notin P), \label{eq:log-out}\\
  \sup_{\substack{m>p\\m\notin P}}\abs{\theta_m}
  +2\sum_{\substack{q\in P\\q>p}}\abs{\theta_q} &<\abs{\theta_p}
       \qquad(p\in P), \label{eq:log-in}
\end{align}
where \(\sup\varnothing=0\).  Then \((\lambda_n)\) has \(\PFNP\).
\end{proposition}

\begin{proof}
We first prove the product rigidity behind the criterion.  Let \(F\subset P\) be finite, let \(m\ge N_0\), let \(\eps\in\{0,1\}\), and let \(\alpha_n,\beta_n\in\{0,1\}\) for \(n\in F\).  Suppose
\begin{equation}\label{eq:rigidity-identity}
  \lambda_m^\eps\prod_{n\in F}\lambda_n^{\alpha_n}
  =\prod_{n\in F}\lambda_n^{\beta_n}.
\end{equation}
Set
\[
  \gamma_k=\eps\mathbf 1_{k=m}+\mathbf 1_{k\in F}(\alpha_k-\beta_k),
\]
where \(\alpha_k=\beta_k=0\) outside \(F\).  Taking logarithms in \eqref{eq:rigidity-identity} gives
\[
  A\log a+B=0,
  \qquad
  A=\sum_k\gamma_k\nu_k\in\Z,
  \qquad
  B=\sum_k\gamma_k\theta_k.
\]
The support of \(\gamma\) is contained in \(F\cup\{m\}\), and each index contributes at most twice.  Hence \eqref{eq:log-small} gives \(\abs{B}<\log a\).  Since \(A\) is an integer, \(A=0\), and then \(B=0\).

Assume that some \(\gamma_k\ne0\), and let \(k_0\) be the smallest such index.  If \(k_0\notin P\), then \(k_0=m\), \(\gamma_{k_0}=1\), and all later non-zero coefficients lie on \(P\) with absolute value at most one.  By \eqref{eq:log-out},
\[
  \sum_{k>k_0}\abs{\gamma_k\theta_k}
  \le \sum_{\substack{p\in P\\p>k_0}}\abs{\theta_p}
  <\abs{\theta_{k_0}},
\]
which contradicts \(B=0\).

If \(k_0\in P\), then later coefficients consist of at most one index outside \(P\), with coefficient one, and indices in \(P\), with absolute coefficients at most two.  By \eqref{eq:log-in},
\[
  \sum_{k>k_0}\abs{\gamma_k\theta_k}
  <\abs{\theta_{k_0}}\le \abs{\gamma_{k_0}\theta_{k_0}},
\]
again contradicting \(B=0\).  Thus \(\gamma_k=0\) for all \(k\).

If \(\eps=0\), this says \(\alpha_n=\beta_n\) for all \(n\in F\).  If \(\eps=1\), then \(\gamma_m=0\) forces \(m\in F\), \(\alpha_m=0\), \(\beta_m=1\), and \(\alpha_n=\beta_n\) for \(n\ne m\).

We now verify \(\PFNP\).  By deleting finitely many initial terms and slightly decreasing \(\eta\), we may use the upper blocks
\[
  E_Q=P\cap(Q/2,Q]
\]
for \(Q\in\cQ\).  The preceding paragraph with \(\eps=0\) gives uniqueness of product layers.  The same paragraph with \(\eps=1\) gives the one-step condition, because all indices in the chosen blocks lie in \(P\).  Therefore \(\PFNP\) holds.
\end{proof}

\begin{proposition}[geometric logarithmic tails]\label{prop:geometric-tail}
In the setting of Proposition~\ref{prop:log-PFNP}, assume that for some \(0<\rho<1\),
\[
  0<\abs{\theta_{n+1}}\le \rho\abs{\theta_n}
  \qquad(n\ge n_0).
\]
Then \((\lambda_n)\) has \(\PFNP\).  Consequently \((\lambda_n)\) is neither \(L^\infty\)-Khintchin nor \(L^1\)-Khintchin.
\end{proposition}

\begin{proof}
Choose \(N_0\ge n_0\) so large that \eqref{eq:log-small} holds.  Choose \(L\ge1\) with
\[
  \frac{\rho}{1-\rho^L}<1,
  \qquad
  \rho+\frac{2\rho^L}{1-\rho^L}<1.
\]
Let \(P\) be one residue class modulo \(L\), restricted to \(n\ge N_0\).  It has positive density in upper blocks.

If \(n\notin P\), then the first element of \(P\) larger than \(n\) is at distance at least one.  Hence
\[
  \sum_{\substack{p\in P\\p>n}}\abs{\theta_p}
  \le \frac{\rho}{1-\rho^L}\abs{\theta_n}
  <\abs{\theta_n}.
\]
If \(p\in P\), then
\[
  \sup_{\substack{m>p\\m\notin P}}\abs{\theta_m}\le \rho\abs{\theta_p},
  \qquad
  2\sum_{\substack{q\in P\\q>p}}\abs{\theta_q}
  \le \frac{2\rho^L}{1-\rho^L}\abs{\theta_p}.
\]
The second choice of \(L\) gives \eqref{eq:log-in}.  Thus Proposition~\ref{prop:log-PFNP} applies, and \cref{thm:PFNP-obstruction} gives the non-Khintchin conclusion.
\end{proof}

\begin{proof}[Proof of \cref{thm:perturbed-polynomial-powers}]
For all sufficiently large \(n\), we have
\[
  \log\lambda_n=p(n)\log a+\theta_n,
  \qquad p(n)\in\Z.
\]
The assumed geometric decay of \(\abs{\theta_n}\) allows us to apply Proposition~\ref{prop:geometric-tail}.  Hence \((\lambda_n)\) is not \(L^\infty\)-Khintchin.  It is therefore not \(L^1\)-Khintchin.

It remains to check the constant perturbation.  Let \(f(n)=c\ne0\).  If \(p\) is eventually constant and the tail \(a^{p(n)}+c\) is positive, then the sequence is eventually constant; choosing a non-zero character shows that the averages cannot converge a.e. to the mean for all bounded functions.  If \(p\) is not eventually constant, then eventual increase of the integer polynomial gives
\[
  p(n+1)-p(n)\ge1
\]
for all large \(n\).  Since
\[
  \theta_n=\log(1+c a^{-p(n)})\sim c a^{-p(n)},
\]
we have
\[
  \frac{\abs{\theta_{n+1}}}{\abs{\theta_n}}
  \sim a^{-(p(n+1)-p(n))}\le a^{-1}<1.
\]
Thus \(\abs{\theta_n}\) eventually satisfies the geometric hypothesis.  After deleting the finite initial part needed for positivity, the conclusion follows from the first part of the theorem.
\end{proof}

\end{document}